\documentclass[12pt]{amsart}

\numberwithin{equation}{section}
\newtheorem{theorem}[equation]{Theorem}
\newtheorem{corollary}[equation]{Corollary}

\newtheorem{lemma}[equation]{Lemma}

\newtheorem*{definition*}{Definition}
\newtheorem*{remark*}{Remark}

\newcommand{\Z}{\mathbb{Z}}

\usepackage{amssymb}
\usepackage{xypic}

\begin{document}

\title[An Elementary (Number Theory) Proof of Touchard]{An Elementary (Number Theory) Proof of Touchard's Congruence}

\author[G.~Hurst]{Greg Hurst}
\address{808 Coventry Point, Springfield IL, \ 62702 \ USA}
\email{ChipH588@aol.com}

\author[A.~Schultz]{Andrew Schultz}
\address{Department of Mathematics, University of Illinois at Urbana-Champaign, 1409 W. Green Street, Urbana, IL \ 61801 \ USA}
\email{acs@math.uiuc.edu}

\begin{abstract}
Let $B_n$ denote the $n$th Bell number. We use well-known recursive expressions for $B_n$ to give a generalizing recursion that can be used to prove Touchard's congruence.
\end{abstract}

\maketitle

\parskip=12pt plus 2pt minus 2pt

\section{Introduction}

For a positive integer $n$, the Bell number $B_n$ is the number of ways a set of $n$ elements can be partitioned into nonempty subsets.  Computations of Bell numbers rely on well known recursive formulae.  For example, one can compute $B_{n+1}$ by enumerating partitions according to the size of the subset $\mathcal{S}$ which contains the $n+1$st element; if $|\mathcal{S}| = n+1-d$ then there are ${n \choose n-d}$ choices for the other elements of $\mathcal{S}$ and $B_d$ ways to partition the elements not in $\mathcal{S}$.  Hence we have \begin{equation}\label{eq:bell_sum}B_{n+1} = \sum_{d=0}^n B_d{n \choose n-d}.\end{equation}

Alternatively, one can compute $B_n$ by keeping track of the number of subsets in a given partition; the Stirling number of the second kind ${n \brace k}$ counts the number of partitions of $n$ elements into precisely $k$ subsets, and so one has the well-known formula \begin{equation}\label{eq:bell_stirling}B_{n} = \sum_{k=1}^n {n \brace k}.\end{equation}  Calculations for ${n \brace k}$ rely on the binomial-like recurrence relation \begin{equation}\label{eq:stirling_recursion}{n+1\brace k}={n\brace k-1}+k{n\brace k}.\end{equation}

Using the identities above, we derive an expression for $B_{n+j}$ which generalizes both (\ref{eq:bell_sum}) and (\ref{eq:bell_stirling}):
\begin{theorem}\label{th}
For positive integers $n$ and $j$, $$B_{n+j} = \sum_{k=1}^n P_j(k){n \brace k},$$ where $P_j(x)$ is the degree $j$ polynomial $ \sum_{r=0}^jB_{j-r}{j \choose r}x^r$.
\end{theorem}
It appears that this formulation wasn't discovered until recently when a combinatorial proof of this result was given in \cite{Spivey}.  We arrived at this result independently through algebraic manipulations of (\ref{eq:bell_sum}-\ref{eq:stirling_recursion}); since this result comes directly from these well-known identities, it is surprising that Theorem \ref{th} wasn't recognized much sooner.  

This project began when the first author was a student in the second author's elementary number theory class.  Before knowing the general form of the polynomials $P_j(k)$ from Theorem \ref{th}, the authors could find specific identities such as $$B_{n+5} = \sum_{k=1}^n (k^5 + 5k^4 + 20k^3 + 50k^2 + 75k + 52){n \brace k}$$
These computations were being carried out just as modular arithmetic was being introduced in the concurrent course, and it was clear that when $j$ was prime the polynomials $P_j(k)$ were ripe for simplification modulo $j$.  Flushing this observation out not only provided a great tour of some of the most familiar identities and techniques for computations modulo $p$, but ultimately led to a rediscovery of 
\begin{corollary}[Touchard's Congruence]\label{cor}
For positive integers $n$ and $m$ and $p$ a prime number, $$B_{n+p^m} \equiv mB_n + B_{n+1} \mod{p}.$$
\end{corollary}

\section{A Proof of Theorem \ref{th}}

We prove the result by induction, with $j=0$ our (trivial) base case.  By induction we have 
$$B_{n+j+1} = \sum_{k=1}^{n+1} P_j(k){n+1 \brace k} =  \sum_{k=1}^{n+1}\sum_{r=0}^j  B_{j-r}{j \choose r}k^r {n+1 \brace k}.$$ Applying identity (\ref{eq:stirling_recursion}) to ${n+1 \brace k}$, and using ${n \brace 0} = {n \brace n+1} = 0$, we have 
\begin{equation*}
\begin{split}
\sum_{k=1}^{n+1}\sum_{r=0}^j& B_{j-r}{j \choose r} k^r\left({n \brace k-1}+k{n \brace k}\right) \\&= \sum_{k=2}^{n+1} \sum_{r=0}^j B_{j-r}{j \choose r}k^r{n \brace k-1} + \sum_{k=1}^{n}\sum_{r=0}^j B_{j-r}{j \choose r}k^{r+1}{n \brace k}\\&=
\sum_{k=1}^{n} \sum_{r=0}^j B_{j-r}{j \choose r}(k+1)^r{n \brace k} + \sum_{k=1}^{n}\sum_{r=0}^j B_{j-r}{j \choose r}k^{r+1}{n \brace k}.
\end{split}
\end{equation*}
The coefficient of $k^\ell {n \brace k}$ in this sum is then 
\begin{equation*}
\begin{split}
\sum_{r=\ell}^jB_{j-r}{j \choose r}&{r \choose \ell}+B_{j-(\ell-1)}{j \choose \ell-1}\\&=
\sum_{r=\ell}^jB_{j-r}{j \choose \ell}{j-\ell \choose r-\ell}+B_{j-\ell+1}{j \choose \ell-1}.
\end{split}
\end{equation*}
We now make the change of variable $d = j-r$ in the first sum and apply identity (\ref{eq:bell_sum}), leaving us with:
\begin{equation*}
\begin{split}
\left(\sum_{d=0}^{j-\ell}B_{d}{j-\ell \choose j-\ell-d}\right)&{j \choose \ell}+B_{j-\ell+1}{j \choose \ell-1} \\&=
B_{j-\ell+1}{j\choose \ell}+B_{j-\ell+1}{j \choose \ell-1} = B_{j-\ell+1}{j+1\choose \ell}.
\end{split}
\end{equation*}
Thus we have $\displaystyle B_{n+j+1} = \sum_{k=1}^n\sum_{r=0}^{j+1}B_{j+1-r}{j+1 \choose r} k^r {n \brace k} = \sum_{k=1}^n P_{j+1}(k){n \brace k}.$

\section{Computations modulo $p$}

If $p$ is an odd prime, it is well known that ${p^m \choose r} \equiv 0 \mod{p}$ whenever $0 < r < p^m$, and so all but two of the terms of $P_{p^m}(x)$ are congruent to zero modulo $p$.  Applying Theorem \ref{th} and Fermat's Little Theorem, we therefore have 
$$B_{n+p^m} \equiv \sum_{k=1}^n (B_{p^m} + k^{p^m}B_0){n\brace k} \equiv \sum_{k=1}^n (B_{p^m} + k){n \brace k} \mod{p}.$$
Since $P_1(k) = 1+k$, Theorem \ref{th} gives $B_{n+1} = \sum (1+k){n \brace k}$, and so rearranging the previous congruence gives
\begin{equation*}\begin{split}B_{n+p^m} &\equiv (B_{p^m}-1)\sum_{k=1}^n{n \brace k} + \sum_{k=1}^n(1+k){n \brace k} 
\\&\equiv (B_{p^m}-1)B_n + B_{n+1} \mod{p}.\end{split}\end{equation*}  The same congruence holds for $p=2$ as well: if $m>2$ then ${2^m \choose r} \equiv 0 \mod{2}$ for $0<r<2^m$ and our previous argument holds, and if $m=2$ the only additional term is $3k^2B_2 = 6k^2 \equiv 0 \mod{2}$.

To verify Corollary \ref{cor}, then, we only need to prove the following
\begin{lemma}
For every positive integer $m$ and prime number $p$, $B_{p^m} \equiv m+1 \mod{p}$.
\end{lemma}
\begin{proof}
$B_{p^m}$ enumerates the partitions of $\Z/p^m\Z$. Our strategy will be to let $\Z/p^m\Z$ act on these partitions in the natural way: for elements $x$ and $y$ of $\Z/p^m\Z$ we define $f_y(x) = x+y \mod{p^m}$, and $f_y(P)$ is the partition we get by applying $f_y$ element-wise to $P$.  Any partition not fixed under this action will belong to an orbit of size a power of $p$, and so the number of fixed partitions is equivalent to $B_{p^m}$ modulo $p$.

So suppose you have some fixed partition $P$ with elements $a$ and $b$ inside subsets $\mathcal{A}$ and $\mathcal{B}$ (respectively).  Then clearly $f_{b-a}(a) = b$, and since $P$ is fixed this means $f_{b-a}(\mathcal{A}) = \mathcal{B}$.  Hence for a fixed partition $P$, all subsets of $P$ must be the same size, and therefore some power of $p$.  We claim that the only fixed partition whose subsets are size $p^j$ is the partion whose subsets contain elements which are congruent to each other modulo $p^{m-j}$.  This would leave us with $m+1$ many fixed partitions, as desired.

To prove the claim, suppose to the contrary that we have a fixed partition $P$ with elements $a,b$ in the same $p^j$-element subset $\mathcal{A}$ which satisfy $a \not\equiv b \mod{p^{m-j}}$.  Now $f_{b-a}(a) = b$, and so $\mathcal{A}$ is permuted by the action of $f_{b-a}$.  Hence for any integer $r$, the $r$-fold composition of the map $f_{b-a}$ --- namely, the map $f_{r(b-a)}$ --- again takes $a$ to some element of $\mathcal{A}$. Now clearly $f_{r(b-a)}(a) \equiv f_{s(b-a)}(a) \mod{p^m}$ if and only if $$a+r(b-a) \equiv a+s(b-a) \mod{p^m}.$$ Since $a \not\equiv b \mod{p^{m-j}}$, however, this congruence forces $r-s \equiv 0 \mod{p^{j+1}}$.  Hence for $r$ between $1$ and $p^{j+1}$, the elements $f_{r(b-a)}(a)$ are distinct contituents of $\mathcal{A}$. It follows that $|\mathcal{A}|>p^j$, a contradiction.
\end{proof}

\section{Acknowledgement}

Both authors would like to thank Bruce Reznick for his advice, guidance, and constant cheer.

\end{document}